\documentclass[a4paper,reqno]{amsart}

\usepackage{amssymb,amsmath}

\usepackage{paralist,enumitem}
\setlist{listparindent=0pt,parsep=3pt}
\setdefaultenum{1.}{}{}{}

\usepackage[pdftex,linktocpage,pdfstartview=FitH,colorlinks=true,allcolors=blue]{hyperref}
\usepackage{amsthm} % put after hyperref to turn off warnings
\usepackage[capitalise]{cleveref}

\usepackage[initials,msc-links]{amsrefs}[2007/10/22]

\newtheorem{thm}{Theorem}[section]
\newtheorem{prop}[thm]{Proposition}
\newtheorem{cor}[thm]{Corollary}
\newtheorem{lem}[thm]{Lemma}
\newtheorem*{proposition*}{Proposition}
\newtheorem*{theorem*}{Theorem}
\newtheorem*{cor*}{Corollary}
\theoremstyle{definition}

\theoremstyle{remark}
\newtheorem*{rem}{Remark}
\newtheorem*{rems}{Remarks}
\newtheorem*{commentary}{Commentary}

\numberwithin{equation}{section}

\newcommand{\R}{\mathbb{R}}
\newcommand{\C}{\mathbb{C}}
\newcommand{\Z}{\mathbb{Z}}

\newcommand{\N}{\mathbb{N}}

\DeclareMathOperator{\rank}{rank}

\DeclareMathOperator{\Aut}{Aut}

\newcommand{\trace}{\operatorname{trace}}
\renewcommand{\d}{\mathop{}\!\mathrm{d}}

\DeclareMathOperator{\rSO}{SO}
\DeclareMathOperator{\rSL}{SL}
\DeclareMathOperator{\rSp}{Sp}
\newcommand{\g}{\mathfrak{g}}

\newcommand{\p}{\mathfrak{p}}

\renewcommand{\k}{\mathfrak{k}}

\newcommand{\Ad}{\operatorname{Ad}}
\renewcommand{\a}{\mathfrak{a}}
\newcommand{\ad}{\operatorname{ad}}
\newcommand{\n}{\mathfrak{n}}

\newcommand{\st}{\mathrel{|}}
\newcommand{\set}[1]{\{#1\}}
\newcommand{\ci}[1]{C^{\infty}(#1)}
\newcommand{\ip}[1]{(#1)}

\newcommand{\vh}{\hat{v}}

\newcommand{\s}{\n\oplus\a}
\renewcommand{\sb}{\n^{\beta}\oplus\a^{\beta}}
\newcommand{\NA}{N\!A}
\newcommand{\NAb}{N^{\beta}\!A^{\beta}}
\let\phi\varphi

\newcommand{\TitleWithUrl}[1]{\IfEmptyBibField{doi}%
  {\IfEmptyBibField{url}{\textit{#1}}%
    {\IfEmptyBibField{eprint}{\href {\BibField{url}}{\textit{#1}}}{\textit{#1}}}%
    }%
  {\href {https://doi.org/\BibField{doi}}{\textit{#1}}}}
\renewcommand{\eprint}[1]{\IfEmptyBibField{url}{\url{#1}}%
  {\href {\BibField{url}}{#1}}}

\BibSpec{article}{%
    +{}  {\PrintAuthors}                {author}
    +{,} { \TitleWithUrl}               {title}
    +{.} { }                            {part}
    +{:} { \textit}                     {subtitle}
    +{,} { \PrintContributions}         {contribution}
    +{.} { \PrintPartials}              {partial}
    +{,} { }                            {journal}
    +{}  { \textbf}                     {volume}
    +{}  { \PrintDatePV}                {date}
    +{,} { \issuetext}                  {number}
    +{,} { \eprintpages}                {pages}
    +{,} { }                            {status}
    +{,} { available at \eprint}        {eprint}
    +{}  { \parenthesize}               {language}
    +{}  { \PrintTranslation}           {translation}
    +{;} { \PrintReprint}               {reprint}
    +{.} { }                            {note}
    +{.} {}                             {transition}
    +{}  {\SentenceSpace \PrintReviews} {review}
}

\BibSpec{collection.article}{%
    +{}  {\PrintAuthors}                {author}
    +{,} { \TitleWithUrl}                     {title}
    +{.} { }                            {part}
    +{:} { \textit}                     {subtitle}
    +{,} { \PrintContributions}         {contribution}
    +{,} { \PrintConference}            {conference}
    +{}  {\PrintBook}                   {book}
    +{,} { }                            {booktitle}
    +{,} { \PrintDateB}                 {date}
    +{,} { pp.~}                        {pages}
    +{,} { }                            {status}
    +{,} { available at \eprint}        {eprint}
    +{}  { \parenthesize}               {language}
    +{}  { \PrintTranslation}           {translation}
    +{;} { \PrintReprint}               {reprint}
    +{.} { }                            {note}
    +{.} {}                             {transition}
    +{}  {\SentenceSpace \PrintReviews} {review}
}
\BibSpec{book}{%
    +{}  {\PrintPrimary}                {transition}
    +{,} { \TitleWithUrl}               {title}
    +{.} { }                            {part}
    +{:} { \textit}                     {subtitle}
    +{,} { \PrintEdition}               {edition}
    +{}  { \PrintEditorsB}              {editor}
    +{,} { \PrintTranslatorsC}          {translator}
    +{,} { \PrintContributions}         {contribution}
    +{,} { }                            {series}
    +{,} { \voltext}                    {volume}
    +{,} { }                            {publisher}
    +{,} { }                            {organization}
    +{,} { }                            {address}
    +{,} { \PrintDateB}                 {date}
    +{,} { }                            {status}
    +{}  { \parenthesize}               {language}
    +{}  { \PrintTranslation}           {translation}
    +{;} { \PrintReprint}               {reprint}
    +{.} { }                            {note}
    +{.} {}                             {transition}
    +{}  {\SentenceSpace \PrintReviews} {review}
}

\BibSpec{thesis}{%
    +{}  {\PrintAuthors}                {author}
    +{.} { \PrintDate}                  {date}
    +{.} { \TitleWithUrl}               {title}
    +{:} { \textit}                     {subtitle}
    +{,} { \PrintThesisType}            {type}
    +{,} { }                            {organization}
    +{,} { }                            {address}
    +{,} { \eprint}                     {eprint}
    +{,} { }                            {status}
    +{}  { \parenthesize}               {language}
    +{}  { \PrintTranslation}           {translation}
    +{;} { \PrintReprint}               {reprint}
    +{.} { }                            {note}
    +{.} {}                             {transition}
    +{}  {\SentenceSpace \PrintReviews} {review}
}

\title[Harmonic submersions between symmetric spaces]{Harmonic
  Riemannian submersions\\ between Riemannian symmetric
  spaces\\ of noncompact type}
\author{F.E. Burstall}
\address{Department of Mathematical Sciences\\ University of Bath\\
  Bath BA2 7AY\\UK}
\email{feb@maths.bath.ac.uk}


\thanks{It is a pleasure to thank Sigmundur Gudmundsson for
  instructive conversations during the preparation of this
  note.  I am also grateful to David Calderbank and John
  C. Wood for helpful comments on an earlier draft.}

\begin{document}

\begin{abstract}
  We construct harmonic Riemannian submersions that are
retractions from symmetric spaces of noncompact type onto
their rank-one totally geodesic subspaces.  Among the
consequences, we prove the existence of a non-constant,
globally defined complex-valued harmonic morphism from the
Riemannian symmetric space associated to a split real
semisimple Lie group.  This completes an affirmative proof
of a conjecture of Gudmundsson.
\end{abstract}

\maketitle

\section{Introduction}
\label{sec:introduction}

A Riemannian symmetric space of noncompact type is a
homogeneous space $M$ of the form $G/K$ with $G$ a connected
noncompact semisimple Lie group with finite centre and $K$ a
maximal compact subgroup \cite{Hel78}*{Chapter V1, \S1}.

The \emph{rank} of $M$ is the maximal dimension of a flat
totally geodesic submanifold.  It is now classical
\cite{Hel78}*{Chapter IX, \S2} that any
Riemannian symmetric space of noncompact type contains
totally geodesic submanifolds which are Riemannian symmetric
spaces of rank one: indeed, there is such a submanifold
$M_{\beta}$ associated to any simple restricted root $\beta$
(see \S\ref{sec:preliminaries} for definitions).

Our main observation, \cref{th:3}, is that there is a
retraction from $M$ onto $M_{\beta}$ that is a harmonic
Riemannian submersion.  This has a number of interesting
applications since such maps intertwine the
Laplace--Beltrami operators of domain and codomain
\cite{Wat73}.

In particular, our retraction is a harmonic morphism (that
is, pulls back germs of harmonic functions to germs of
harmonic functions).  As a corollary, we are able to
complete the affirmation of a
long-standing\footnote{Gudmundsson first conjectured this
  result in the late nineties (private communication).}
conjecture of Gudmundsson:
\begin{theorem*}[\S\ref{sec:harmonic-morphisms}]
  Let $M$ be a Riemannian symmetric space of noncompact
  type.  Then there is a non-constant, globally defined
  harmonic morphism from $M$ to $\C$.
\end{theorem*}

Again, pullback by our retraction preserves the class of
\emph{eigenfunctions in the sense of Gudmundsson--Sobak
  \cite{GudSob20}}: these are complex-valued functions $f$
for which both $f$ and $f^2$ are eigenvectors of the
Laplace--Beltrami operator.  Substantial effort has been
recently made to find such eigenfunctions on Riemannian
symmetric spaces
\cite{MR2395191,MR4632822,MR4626314,GudSob20,MR4382667}.
We exploit the work of Ghandour--Gudmundsson
\cite{MR4632822,MR4626314} on the rank-one case and prove:
\begin{theorem*}[\S\ref{sec:eigenfunctions}]
  Let $M$ be a Riemannian symmetric space of noncompact type
  which is not a product of Cayley hyperbolic planes.  Then
  there exist $f\colon M\to\C$ such that both $f$ and $f^2$
  are eigenvectors of the Laplace--Beltrami operator of $M$.
\end{theorem*}

Finally, for $r\in\Z^{+}$, pullback by our retraction
preserves the class of \emph{proper $r$-harmonic functions},
thus complex-valued functions in the kernel of the $r$-th
power of the Laplace--Beltrami operator but not that of the
$(r-1)$-th power.  Gudmundsson--Siffert--Sobak
\cite{MR4230531} find examples of these on rank-one
symmetric spaces of noncompact type and so we conclude:
\begin{theorem*}[\S\ref{sec:proper-r-harmonic}]
  Let $M$ be a Riemannian symmetric space of noncompact
  type.  Then there are proper $r$-harmonic functions
  $M\to\C$ for every $r\in\Z^+$.
\end{theorem*}

\section{Preliminaries}
\label{sec:preliminaries}

\subsection{Structure theory}
\label{sec:structure-theory}

We begin by setting up the structure theory of noncompact
semisimple groups.  For details (and much more), we refer to
\cite[Chapter~VI]{Hel78}.

Let $M$ be a Riemannian symmetric space of noncompact type
with isometry group $G$.  Fix a base-point $o\in M$ with
stabiliser $K$ so that $K$ is a maximal compact subgroup of
$G$ and $M\cong G/K$.

Let $\g,\k$ be the Lie algebras of $G,K$ and
$\theta\in\Aut(\g)$ the Cartan involution of $\g$ with fixed
set $\k$.  We have the corresponding Cartan decomposition
\begin{equation*}
  \g=\k\oplus\p
\end{equation*}
into $\pm 1$-eigenspaces of $\theta$.

We define a $K$-invariant inner product $\ip{\cdot,\cdot}$
on $\g$ by
\begin{equation*}
  \ip{X,Y}=-B(X,\theta Y).
\end{equation*}
where $B$ is an $\Ad G$-invariant symmetric bilinear
form\footnote{The Killing form of $\g$ will do but if $\g$
  has more than one simple factor, there are many other
  possibilities.} on $\g$ which is positive-definite on $\p$
and negative-definite on $\k$.

Now fix a maximal abelian subspace $\a\leq\p$ (so that
$\dim\a=\rank M$) and let $\Sigma\subset\a^{*}$ be the
restricted roots (thus common eigenvalues of $\ad H$,
$H\in\a$) with restricted root spaces $\g_{\alpha}$,
$\alpha\in\Sigma$:
\begin{equation*}
  \g_{\alpha}=\set{X\in\g\st [H,X]=\alpha(H)X, \text{ for
      all $H\in\a$}}.
\end{equation*}
This gives an orthogonal decomposition
\begin{equation*}
  \g=\g_0\oplus\bigoplus_{\alpha\in\Sigma}\g_{\alpha}.
\end{equation*}
For $\alpha\in\Sigma$, the \emph{multiplicity} of
$\alpha$ is $m_{\alpha}:=\dim\g_{\alpha}$.

Fix a choice of positive restricted roots
$\Sigma^+\subset\Sigma$ and set
\begin{equation*}
  \n=\bigoplus_{\alpha\in\Sigma^+}\g_{\alpha}.
\end{equation*}
We then have the Iwasawa decomposition:
\begin{equation*}
  \g=\n\oplus\a\oplus\k.
\end{equation*}
Let $A,N$ denote the analytic subgroups of $G$ corresponding
to $\a,\n$.  Then multiplication gives a diffeomorphism
$N\times A\times K\to G$ yielding the global Iwasawa
decomposition $G=\NA K$.

\subsection{Simple restricted roots}
\label{sec:simple-restr-roots}

The datum of positive restricted roots $\Sigma^{+}$ leads to
the \emph{simple} restricted roots (roots $\beta\in\Sigma^+$
that cannot be written as a sum of two other positive
roots).  These comprise a basis of $\a^{*}$ and any positive
restricted root can be written uniquely as an $\N$-linear
combination of simple restricted roots.

Our constructions will start with a simple restricted root.
We collect some simple facts about these that we shall rely
on below:
\begin{lem}
  \label{th:1}
  Let $\beta\in\Sigma^+$ be a simple root and set
  $\Sigma^+_{\beta}=\Sigma^{+}\setminus\set{\beta,2\beta}$.
  Then
  \begin{compactenum}
  \item \label{item:1}Let $\n(\beta)\leq\n$ be given by
    \begin{equation*}
      \n(\beta)=\sum_{\alpha\in\Sigma^+_{\beta}}\g_{\alpha}.
    \end{equation*}
    Then $\n(\beta)$ is an ideal of $\n$.
  \item \label{item:2}We have:
    \begin{equation}
      \label{eq:1}
      \sum_{\alpha\in\Sigma^+_{\beta}}m_{\alpha}\ip{\alpha,\beta}=0.
    \end{equation}
  \item \label{item:3} If $m_{\beta}$ is odd then $m_{2\beta}=0$.
  \end{compactenum}
\end{lem}
\begin{proof}
  For the first two assertions, we argue as in
  \cite{Hum72}*{\S10.2}. The key observation is that any
  $\alpha$ lies in $\Sigma^+_{\beta}$ if and only if it has
  some strictly positive coefficient with respect to another
  simple root.  This property is unchanged when any positive
  restricted root or any multiple of $\beta$ is added.
  Since
  $[\g_{\alpha},\g_{\alpha'}]\leq \g_{\alpha+\alpha'}$, for
  any $\alpha,\alpha'\in\Sigma$, this settles item
  \ref{item:1}.  Moreover, it shows that the root reflection
  $\sigma_{\beta}$ must permute $\Sigma_{\beta}^+$,
  preserving multiplicities, and so must fix
  $\sum_{\alpha\in\Sigma_{\beta}^+}m_{\alpha}\alpha$.  Thus
  \eqref{eq:1} follows.

  Item \ref{item:3} is due to Araki
  \cite{MR153782}*{Proposition~2.3}.
\end{proof}
\subsection{Rank-one symmetric subspaces}
\label{sec:rank-one-symmetric}

Let $\beta\in\Sigma^+$ be a simple restricted root and
contemplate the $\theta$-stable Lie subalgebra
$\g^{\beta}\leq\g$ generated by $\g_{\pm\beta}$.  Define
subalgebras of $\g^{\beta}$ by
\begin{equation*}
  \n^{\beta}=\g_{\beta}\oplus\g_{2\beta},\qquad
  \k^{\beta}=\g^{\beta}\cap\k,\qquad
  \a^{\beta}=\R H_{\beta}
\end{equation*}
where $H_{\beta}\in\a$ is determined by
$\beta(H)=\ip{H_{\beta},H}$, for all $H\in\a$.  Further, let
$G^{\beta},K^{\beta},N^{\beta},A^{\beta}$ be the
corresponding analytic subgroups of $G$.  We have
\begin{prop}[\cite{Hel78}*{Chapter IX, \S2}]
  \label{th:2}
  $G^{\beta}$ is a semisimple Lie group with Iwasawa
  decomposition $N^{\beta}\!A^{\beta}K^{\beta}$.  Moreover,
  \begin{equation*}
    K^{\beta}=G^{\beta}\cap K
  \end{equation*}
  so that the symmetric space $M_{\beta}:=G^{\beta}/K^{\beta}$ embeds
  in $M=G/K$ totally geodesically as the $G^{\beta}$-orbit
  of $o$.
\end{prop}
\begin{rems}
\item[]
  \begin{compactenum}
  \item $M_{\beta}$ is a rank-one symmetric space of
    dimension $1+m_{\beta}+m_{2\beta}$.  We can detect the
    isomorphism type of $M_{\beta}$ from $m_{2\beta}$:
    according to whether $m_{2\beta}=0,1$ or $3$,
    $M_{\beta}$ is homothetic to a real, complex or
    quaternionic hyperbolic space.  Exceptionally, one has
    $m_{2\beta}=7$ which only occurs when $M=M_{\beta}$ is
    the hyperbolic Cayley plane $\mathbb{O}H^2$.  See
    \cite{MR153782}*{\S5.11} for the compete list of simple
    restricted root multiplicities for each simple
    noncompact $\g$.
  \item In particular, when $m_{\beta}=1$ (which is always
    the case when $\g$ is the split real form of $\g^{\C}$),
    we have $m_{2\beta}=0$, by \cref{th:1}(\ref{item:3}),
    and $M_{\beta}$ is isometric to a hyperbolic plane.
  \item On the other hand, if $M$ is already rank-one, then
    $M_{\beta}=M$!
  \item Finally, we can understand the scaling of the metric
    on $M_{\beta}$: the minimum sectional curvature of
    $M_{\beta}$ is $-\ip{\beta,\beta}$, c.f.\
    \cite{Hel66}*{Theorem 1.1}.
  \end{compactenum}
\end{rems}

Our mission is to prove that there is a harmonic Riemannian
submersion $M\to M_{\beta}$ and it is to this that we now turn.

\section{Harmonic Riemannian submersions}
\label{sec:harm-riem-subm}

We identify $\NA$ with $M$ via $na\mapsto nao$.  The
Riemannian metric on $M$ induced by $B$ then pulls back to
the left invariant metric on $\NA$ induced by
$\ip{\cdot,\cdot}$ on $\s$.  In the same way, we have an
isometry $\NAb\cong M_{\beta}$.

Now let $\pi\colon \s\to\sb$ be orthogonal projection and remark
that
\begin{equation*}
  \ker\pi=\n(\beta)\oplus\ker\beta
\end{equation*}
is an ideal of $\s$ by \cref{th:1}(\ref{item:1}) so that
$\pi$ is a Lie algebra homomorphism.  Since $\NA$ is simply
connected, we integrate to get a Lie group homomorphism
$\NA\to\NAb$, also called $\pi$.

View $X\in\s$ or $\sb$ as a left-invariant vector field on
$\NA$ or $\NAb$ according to context.  Then
\begin{equation*}
  \d\pi(X)=\pi(X),
\end{equation*}
for $X\in\s$, so that $\pi\colon \NA\to\NAb$ is a Riemannian
submersion.

We want to show that $\pi$ is a harmonic map which means
that
\begin{equation*}
  \tau_{\pi}:=(\pi^{-1}\nabla^{\beta})_{e_k}\d\pi(e_k)-\d\pi(\nabla_{e_k} e_k)=0,
\end{equation*}
where $e_k$ is an orthonormal frame of $T\NA$ and
$\nabla,\nabla^{\beta}$ are the Levi-Civita connections of
$\NA,\NAb$.

To compute this, choose orthonormal bases $(e_{i})_{i}$ of
$\ker\pi$ and $(e_{j})_{j}$ of $\sb$ and combine them to get a
left-invariant orthonormal frame  of $\NA$.  Now, for
$X\in\sb$, 
\begin{equation*}
  \ip{\tau_{\pi},X}=\ip{\nabla_{e_{j}}e_j,X}-\ip{\nabla_{e_k}e_k,X}=-\ip{\nabla_{e_i}e_i,X}
\end{equation*}
which reads, in view of the Koszul formula,
\begin{equation*}
  \ip{e_i,[e_i,X]}=-\trace_{\ker\pi}\ad
  X.
\end{equation*}
Thus, $\pi$ is harmonic exactly when each $\ad X$, $X\in\sb$, is
trace-free on $\ker\pi$.  However, when $X\in\n^{\beta}$,
$\ad X$ is nilpotent and so certainly trace-free.  This leaves $\ad
H_{\beta}$ which has eigenvalue $\ip{\alpha,\beta}$ on any
$\g_{\alpha}$ and so
\begin{equation*}
  \trace_{\ker\pi}\ad H_{\beta}=\sum_{\alpha\in\Sigma^+_{\beta}}m_{\alpha}\ip{\alpha,\beta}=0
\end{equation*}
by \eqref{eq:1}.

To summarise:
\begin{thm}
  \label{th:3}
  $\pi\colon NA\to\NAb$ is a harmonic Riemannian submersion.
\end{thm}
\begin{rems}
\item{}
  \begin{compactenum}
  \item Since $\pi\colon\s\to\sb$ is a projection,
    $\pi\colon\NA\to\NAb$ is a retraction.
  \item $\pi\colon\NA\to\NAb$ is topologically trivial:
    multiplication $m\colon\ker\pi\times\NAb\to\NA$ is a
    diffeomorphism and $\pi=\pi_1\circ m^{-1}$.
  \item According to Watson \cite{Wat73}*{Theorem~2.1}, the
    fibres of $\pi$ are minimal.  However, they are not
    totally geodesic in general.  Indeed, if $M$ is
    irreducible of rank at least two, there is a restricted
    root $\alpha\in\Sigma^+_{\beta}$ with
    $\ip{\alpha,\beta}\neq 0$.  Then, for
    $X\in\g_{\alpha}\leq\n(\beta)$ non-zero, the Koszul
    formula gives:
    \begin{equation*}
      \ip{\nabla_XX,H_{\beta}}=\ip{\alpha,\beta}\ip{X,X}\neq 0.
    \end{equation*}
  \end{compactenum}
\end{rems}

\section{Applications}
\label{sec:applications}

It is a result of Watson \cite{Wat73}*{Theorem~2.1} that a
surjection $\pi\colon M\to B$ of Riemannian manifolds is a
harmonic Riemannian submersion if and only if it intertwines
the Laplace--Beltrami operators of $M$ and $B$:
\begin{equation*}
  (\Delta^Bf)\circ\pi=\Delta^M(f\circ\pi),
\end{equation*}
for all $f\in\ci{B}$.

In particular, harmonic Riemannian submersions enjoy the
following properties:
\begin{compactenum}
\item they are harmonic morphisms: that is, they pull back
  (germs of) harmonic functions to harmonic functions;
\item more generally, they pullback eigenvectors of
  $\Delta^B$ to eigenvectors (with the same eigenvalue) of
  $\Delta^M$;
\item as a result, they pullback eigenfunctions in the sense
  of Gudmundsson--Sobak \cite{GudSob20}*{Definition~2.3}:
  these are eigenvectors $f$ of the Laplace--Beltrami
  operator for which $f^2$ is also an eigenvector;
\item they pullback proper $r$-harmonic functions: these are
  functions $f$ such that $\Delta^rf=0$ while
  $\Delta^{r-1}f\neq 0$.
\end{compactenum}

We now use \cref{th:3} to provide examples of complex-valued
functions of these various types on all (or nearly all)
Riemannian symmetric spaces of noncompact type.

\subsection{Harmonic morphisms}
\label{sec:harmonic-morphisms}

We have:
\begin{cor}
  \label{th:4}
  Let $M$ be a Riemannian symmetric space of noncompact type
  which admits a simple restricted root $\beta$ with $m_{\beta}=1$.
  Then there is a non-constant harmonic morphism $\phi\colon M\to\C$.
\end{cor}
\begin{proof}
  In this case, as we have remarked above, $M_{\beta}$ is a
  hyperbolic plane $H^2$.  Let $f\colon M_{\beta}\to\C$ be
  any holomorphic function (for example, the identification
  of $H^2$ with the upper half-plane).  Then $f$ is a
  harmonic morphism \cite{MR2044031}*{Example~4.2.7} and
  harmonic morphisms are clearly closed under composition so
  that $\phi:=f\circ\pi$ is our desired harmonic morphism.
\end{proof}

By contrast, Gudmundsson--Svensson \cite{MR2525933}*{Example
12.2} prove the existence of a non-constant harmonic
morphism $M\to\C$ when $M$ has a simple restricted root of
multiplicity at least $2$.  Thus, taken together with
\cref{th:4}, we have the following theorem which has long
been conjectured by Gudmundsson (c.f.\
\cite{Sve04}*{Conjecture~7.1.1}):
\begin{thm}
  \label{th:5}
  Let $M$ be a Riemannian symmetric space of noncompact
  type.  Then there is a non-constant harmonic morphism
  $\phi\colon M\to\C$.
\end{thm}
\goodbreak

\begin{commentary}
\item[]
  \begin{compactenum}
  \item There is a lot of prior art for \cref{th:5}.  In
    view of the multiplicity two result of
    Gudmundsson--Svensson, one only needs to consider the
    case where all restricted simple roots have multiplicity
    one.  This occurs exactly when $\g$ is the split real
    form of $\g^{\C}$, or, equivalently, $\rank M=\rank G$,
    so that $\a^{\C}$ is a Cartan subalgebra of $\g^{\C}$.
    However, another result of Gudmundsson--Svensson
    \cite{MR2525933}*{Theorem~11.3} also provides harmonic
    morphisms when $\rank M\geq3$ which only leaves
    $\rSL(3,\R),\rSO(2,3)\cong\rSp(4,\R),\rSO(2,2),G_2^{2}$.
    For these, Gudmundsson--Svensson construct
    complex-valued harmonic morphisms when $G=\rSL(n,\R)$
    \cite{MR2271193}*{Theorem~4.1} while the other two
    classical symmetric spaces are Hermitian symmetric and
    so bounded symmetric domains with plenty of holomorphic
    functions.  In short, the only irreducible $M$ for which
    \cref{th:5} is new is $G^2_2/\rSO(4)$ which has resisted
    all previous approaches to the problem!
  \item It is not difficult to work through the
    identifications to produce explicit formula for the
    harmonic morphisms we have been discussing.  First, when
    $m_{\beta}=1$, $\phi\colon M\to\C$ can be taken to be
    given by
    \begin{equation}\label{eq:2}
      \phi(nao)=\ip{X,\log_Nn}+ie^{\beta(\log_Aa)},
    \end{equation}
    where $X\in\g_{\beta}$ with $\ip{X,X}=\ip{\beta,\beta}$.
    Here $\log_N\colon N\to\n$, $\log_A\colon A\to\a$ invert
    the exponential map on $N$, $A$.

    A similar formula is available when $m_{\beta}\geq2$: in
    this case, one can find $X\in\g_{\beta}^{\C}$ with
    $\ip{X,X}=0$ and then take
    \begin{equation*}
      \phi(nao)=\ip{X,\log_Nn}.
    \end{equation*}
    This reproduces the maps found in \cite{MR2525933}.  It
    is an amusing exercise to verify directly that these
    maps are harmonic morphisms.
  \item There is an alternative approach to \cref{th:4}
    via harmonic analysis on $M$.  When $m_{\beta}=1$, one
    can find a finite-dimensional subrepresentation $V\leq\ci{M}$ of the
    regular representation of $G$ on $\ci{M}$ with the
    following properties: for $v\in V$ the highest weight
    vector with respect to a suitably chosen Borel
    subalgebra of $\g^{\C}$ and $Y\in\g_{-\beta}$, setting
    $u=Yv$ and $\vh=Yu$, we can arrange that both $v$ and
    $2v\vh-u^2$ are strictly positive functions while
    $Y\vh=0$.  Then
    \begin{equation*}
      \phi:=\frac{-u+i\sqrt{2v\vh-u^2}}{v}
    \end{equation*}
    yields a harmonic morphism onto the upper half-plane
    which coincides with \eqref{eq:2} up to scale.  This is an
    abstraction of the argument used in \cite{MR2271193} to
    treat $\rSL(n,\R)/\rSO(n)$.
  \end{compactenum}
\end{commentary}

\subsection{Eigenfunctions}
\label{sec:eigenfunctions}

Ghandour--Gudmundsson
\citelist{\cite{MR4632822}*{Theorems~1.14,
    1.19}\cite{MR4626314}*{Theorem~4.3}} construct
eigenfunctions in the sense of Gudmundsson--Sobak on real,
complex and quaternionic Grassmannians and their non-compact
duals.  In particular, they find eigenfunctions on the real,
complex and quaternionic hyperbolic spaces.  We therefore
conclude:
\begin{thm}
  \label{th:6}
  Let $M$ be a Riemannian symmetric space of noncompact type
  which is not a product of Cayley hyperbolic planes.  Then
  there exist $f\colon M\to\C$ such that both $f$ and $f^2$
  are eigenvectors of $\Delta^M$.
\end{thm}
\begin{proof}
  For such an $M$, there is at least one $M_{\beta}$
  isometric to a real, complex or quaternionic hyperbolic
  space.  Apply \cref{th:3} to pullback the examples of
  Ghandour--Gudmundsson on $M_{\beta}$ to $M$.
\end{proof}

\begin{rem}
  In fact, one can do better than this and construct many
  such eigenfunctions on \emph{any} Riemannian symmetric
  space of semisimple type.  We shall return to this topic
  elsewhere.
\end{rem}

\subsection{Proper $r$-harmonic functions}
\label{sec:proper-r-harmonic}

Gudmundsson--Siffert--Sobak \cite{MR4230531}*{Theorems~4.4,
  4.8} construct proper $r$-harmonic functions on any
rank-one Riemannian symmetric space of non-compact type and
so any $M_{\beta}$.  Thus \cref{th:3} yields
\begin{thm}
  \label{th:7}
  Let $M$ be a Riemannian symmetric space of noncompact
  type.  Then there are proper $r$-harmonic functions
  $M\to\C$, for any $r\in\Z^{+}$.
\end{thm}

% \bibliography{master}

\begin{bibdiv}
  \begin{biblist}
 \bib{MR153782}{article}{
  author={Araki, Sh\^{o}r\^{o}},
  title={On root systems and an infinitesimal classification of irreducible symmetric spaces},
  journal={J. Math. Osaka City Univ.},
  volume={13},
  date={1962},
  pages={1--34},
  issn={0449-2773},
  review={\MR {153782}},
  url={https://projecteuclid.org/journals/journal-of-mathematics-osaka-city-university/volume-13/issue-1/On-root-systems-and-an-infinitesimal-classification-of-irreducible-symmetric/ojm/1353055009.full?tab=ArticleLink},
}

\bib{MR2044031}{book}{
  author={Baird, Paul},
  author={Wood, John C.},
  title={Harmonic morphisms between Riemannian manifolds},
  series={London Mathematical Society Monographs. New Series},
  volume={29},
  publisher={The Clarendon Press, Oxford University Press, Oxford},
  date={2003},
  pages={xvi+520},
  isbn={0-19-850362-8},
  review={\MR {2044031}},
  doi={10.1093/acprof:oso/9780198503620.001.0001},
}

\bib{MR4626314}{article}{
  author={Ghandour, Elsa},
  author={Gudmundsson, Sigmundur},
  title={Explicit $p$-harmonic functions on the real Grassmannians},
  journal={Adv. Geom.},
  volume={23},
  date={2023},
  number={3},
  pages={315--321},
  issn={1615-715X},
  review={\MR {4626314}},
  doi={10.1515/advgeom-2023-0015},
}

\bib{MR4632822}{article}{
  author={Ghandour, Elsa},
  author={Gudmundsson, Sigmundur},
  title={Explicit harmonic morphisms and $p$-harmonic functions from the complex and quaternionic Grassmannians},
  journal={Ann. Global Anal. Geom.},
  volume={64},
  date={2023},
  number={2},
  pages={Paper No. 15, 18},
  issn={0232-704X},
  review={\MR {4632822}},
  doi={10.1007/s10455-023-09919-8},
}

\bib{MR2395191}{article}{
  author={Gudmundsson, Sigmundur},
  author={Sakovich, Anna},
  title={Harmonic morphisms from the classical compact semisimple Lie groups},
  journal={Ann. Global Anal. Geom.},
  volume={33},
  date={2008},
  number={4},
  pages={343--356},
  issn={0232-704X},
  review={\MR {2395191}},
  doi={10.1007/s10455-007-9090-8},
}

\bib{MR4230531}{article}{
  author={Gudmundsson, Sigmundur},
  author={Siffert, Anna},
  author={Sobak, Marko},
  title={Explicit $p$-harmonic functions on rank-one Lie groups of Iwasawa type},
  journal={J. Geom. Phys.},
  volume={164},
  date={2021},
  pages={Paper No. 104205, 16},
  issn={0393-0440},
  review={\MR {4230531}},
  doi={10.1016/j.geomphys.2021.104205},
}

\bib{MR4382667}{article}{
  author={Gudmundsson, Sigmundur},
  author={Siffert, Anna},
  author={Sobak, Marko},
  title={Explicit proper $p$-harmonic functions on the Riemannian symmetric spaces $\bold {SU}(n)/\bold {SO} (n)$, $\bold {Sp}(n)/\bold {U}(n)$, $\bold {SO} (2n)/\bold {U}(n)$, $\bold {SU}(2n)/\bold {Sp}(n)$},
  journal={J. Geom. Anal.},
  volume={32},
  date={2022},
  number={5},
  pages={Paper No. 147, 16},
  issn={1050-6926},
  review={\MR {4382667}},
  doi={10.1007/s12220-022-00885-4},
}

\bib{GudSob20}{article}{
  author={Gudmundsson, Sigmundur},
  author={Sobak, Marko},
  title={Proper $r$-harmonic functions from Riemannian manifolds},
  journal={Ann. Global Anal. Geom.},
  volume={57},
  date={2020},
  number={1},
  pages={217--223},
  issn={0232-704X},
  review={\MR {4057458}},
  doi={10.1007/s10455-019-09696-3},
}

\bib{MR2271193}{article}{
  author={Gudmundsson, Sigmundur},
  author={Svensson, Martin},
  title={On the existence of harmonic morphisms from certain symmetric spaces},
  journal={J. Geom. Phys.},
  volume={57},
  date={2007},
  number={2},
  pages={353--366},
  issn={0393-0440},
  review={\MR {2271193}},
  doi={10.1016/j.geomphys.2006.03.008},
}

\bib{MR2525933}{article}{
  author={Gudmundsson, Sigmundur},
  author={Svensson, Martin},
  title={Harmonic morphisms from solvable Lie groups},
  journal={Math. Proc. Cambridge Philos. Soc.},
  volume={147},
  date={2009},
  number={2},
  pages={389--408},
  issn={0305-0041},
  review={\MR {2525933}},
  doi={10.1017/S0305004109002564},
}

\bib{Hel66}{article}{
  author={Helgason, Sigurdur},
  title={Totally geodesic spheres in compact symmetric spaces},
  journal={Math. Ann.},
  volume={165},
  date={1966},
  pages={309--317},
  issn={0025-5831},
  review={\MR {210043}},
  doi={10.1007/BF01344015},
}

\bib{Hel78}{book}{
  author={Helgason, Sigurdur},
  title={Differential geometry, {L}ie groups, and symmetric spaces},
  series={Pure and Applied Mathematics},
  publisher={Academic Press Inc. [Harcourt Brace Jovanovich Publishers]},
  address={New York},
  date={1978},
  volume={80},
  isbn={0-12-338460-5},
  review={\MR {MR514561 (80k:53081)}},
}

\bib{Hum72}{book}{
  author={Humphreys, James~E.},
  title={Introduction to {L}ie algebras and representation theory},
  publisher={Springer-Verlag},
  address={New York},
  date={1972},
  note={Graduate Texts in Mathematics, Vol. 9},
  review={\MR {MR0323842 (48 \#2197)}},
}

\bib{Sve04}{unpublished}{
  author={Svensson, Martin},
  title={Harmonic morphisms, Hermitian structures and symmetric spaces},
  date={2004},
  note={Lund Ph.D. thesis},
}

\bib{Wat73}{article}{
  author={Watson, B.},
  title={Manifolds maps commuting with the {L}aplacian},
  date={1973},
  journal={J. Differential Geom.},
  volume={8},
  pages={85\ndash 94},
}
  \end{biblist}
\end{bibdiv}
% \bibliographystyle{amsrefs}

\end{document}